\documentclass[11pt,reqno]{amsart}
\usepackage{amsmath,amssymb,mathtools}
\usepackage{booktabs,array}
\usepackage[hidelinks]{hyperref}
\usepackage[margin=1.15in]{geometry}
\allowdisplaybreaks

\newtheorem{theorem}{Theorem}[section]
\newtheorem{proposition}[theorem]{Proposition}
\newtheorem{lemma}[theorem]{Lemma}

\theoremstyle{remark}
\newtheorem{remark}[theorem]{Remark}

\newcommand{\Zp}{\mathbb{Z}_p}

\newcommand{\Gamp}{\Gamma_p}

\newcommand{\om}{\omega}

\title[Deformation and Symmetry]
{Deformation and Symmetry in Two Hypergeometric Supercongruences}
\author{Dianbin Bao}
\date{October 2026}

\begin{document}

\begin{abstract}
	We prove two supercongruence conjectures of Li--Long--Tu concerning
	truncated very-well-poised ${}_7F_6(1)$ series and Fourier coefficients of
	modular forms.  The proofs are based on a common deformation method.  By
	deforming the original truncated hypergeometric series to suitable terminating
	series, we are able to apply classical hypergeometric transformations together
	with Mishev's $W(D_5)$ symmetry.  This reduces both problems to congruences
	for lower-rank hypergeometric sums that can be treated using the $p$-adic
	methods of Long--Tu--Yui--Zudilin.  Although the two cases require different
	local arithmetic calculations, their global proof structure is the same.
	The resulting argument gives a unified explanation of these two
	supercongruences and suggests that the same deformation-and-symmetry method
	may apply more broadly.
\end{abstract}
\maketitle

\section{Introduction}

Arithmetic aspects of hypergeometric functions arise in several closely related
settings.  Truncated hypergeometric series play an important role in $p$-adic
questions and supercongruences, while finite-field hypergeometric functions and
hypergeometric motives connect hypergeometric data with point counts, Galois
representations, and modular forms.  In particular, Katz's theory of
hypergeometric sheaves gives $\ell$-adic Galois representations attached to
hypergeometric data \cite{MR1081536}.  Relations between hypergeometric
functions and modular forms appear in both archimedean and $p$-adic settings:
untruncated hypergeometric values can be related to periods or special
$L$-values of modular forms, while truncated hypergeometric series often
satisfy supercongruences involving $p$-adic quantities or Fourier coefficients
of modular forms; see, for example,
\cite{MR3451938,MR3848417,MR4330088} and the references therein.

Deformation methods for truncated hypergeometric series have proved
effective in earlier work on supercongruences; see, for example,
Long--Ramakrishna \cite{MR3451938} and
Deines--Fuselier--Long--Swisher--Tu \cite{MR3596579}.  In this circle of work,
$p$-adic deformations are combined with classical hypergeometric identities
and $p$-adic gamma techniques.  Our argument follows this general
philosophy, with Mishev's $W(D_5)$ symmetry providing the additional
structure that connects the terminating ${}_7F_6(1)$ deformation to the
rank-four hypergeometric sums.

Li, Long, and Tu \cite{MR4445932} revisited Whipple's classical
${}_7F_6(1)$ transformation from the viewpoint of hypergeometric data and,
among other results, recorded a number of numerical supercongruence
conjectures.  Two of these conjectures concern truncated very-well-poised
${}_7F_6(1)$ series and congruences modulo $p^3$ with Fourier coefficients of
the modular forms $f_{12.4.a.a}$ and $f_{24.4.a.a}$. The purpose of the present paper is to prove these two conjectures.

A crucial lower-rank input is provided by the work of
Long--Tu--Yui--Zudilin \cite{MR4330088} on rigid hypergeometric
Calabi--Yau threefolds.  Their results identify the corresponding rank-four
truncated ${}_4F_3(1)$ series modulo $p^3$ with Fourier coefficients of
modular forms.  Thus, for the two cases considered here, the remaining problem
is to construct a congruence-level bridge from the Li--Long--Tu
${}_7F_6(1)$ truncations to these rank-four sums.
\section{Main Results and Hypergeometric Reduction}\label{sec:main-results}

For $a$ in a field of characteristic zero and $k\geq 0$, write
\[
(a)_0:=1,\qquad
(a)_k:=a(a+1)\cdots(a+k-1)
\]
for the Pochhammer symbol.  For complex parameters we use the standard notation
\[
{}_rF_s\!\left[
\begin{matrix}
	a_1,\dots,a_r\\
	b_1,\dots,b_s
\end{matrix};z
\right]
:=
\sum_{k=0}^{\infty}
\frac{(a_1)_k\cdots(a_r)_k}
{(b_1)_k\cdots(b_s)_k}
\frac{z^k}{k!},
\]
whenever the series is defined.  Its truncation at $m$ is denoted by
\[
{}_rF_s\!\left[
\begin{matrix}
	a_1,\dots,a_r\\
	b_1,\dots,b_s
\end{matrix};z
\right]_m
:=
\sum_{k=0}^{m}
\frac{(a_1)_k\cdots(a_r)_k}
{(b_1)_k\cdots(b_s)_k}
\frac{z^k}{k!}.
\]

Let $p\geq3$ be prime.  For $\rho\in\{\frac13,\frac16\}$ define
\[
H_{p,\rho}
:=
{}_7F_6\!\left[
\begin{matrix}
	\frac12,\frac54,\frac12,\frac12,\frac12,\rho,1-\rho\\
	\frac14,1,1,1,\frac32-\rho,\frac12+\rho
\end{matrix};1
\right]_{p-1},
\]
and
\[
S_{p,\rho}
:=
{}_4F_3\!\left[
\begin{matrix}
	\frac12,\frac12,\rho,1-\rho\\
	1,1,1
\end{matrix};1
\right]_{p-1}.
\]
If $f=\sum_{n\geq1}a_n(f)q^n$ is a normalized modular form, we write
$a_p(f)$ for its $p$th Fourier coefficient.  We follow Li--Long--Tu in
using LMFDB labels for the modular forms below.

The following theorem proves two numerical supercongruence conjectures of
Li--Long--Tu \cite{MR4445932}.

\begin{theorem}\label{thm:main}
	For every prime $p\geq3$,
	\[
	H_{p,1/3}\equiv a_p(f_{12.4.a.a})\pmod{p^3},
	\]
	and
	\[
	H_{p,1/6}\equiv a_p(f_{24.4.a.a})\pmod{p^3}.
	\]
\end{theorem}

These two congruences were stated by Li--Long--Tu as part of a larger
collection of numerical supercongruence conjectures, for which they made
the uniform choice $p>11$ \cite{MR4445932}.  The cases $p=3$ and $p=5$
are verified directly from the defining finite sums and the corresponding
Fourier coefficients.  Henceforth we assume $p\geq7$.

The corresponding rank-four input is supplied by Long--Tu--Yui--Zudilin
\cite{MR4330088}.  For these two rigid hypergeometric data, one has
\[
S_{p,1/3}\equiv a_p(f_{36.4.a.a})\pmod{p^3},
\qquad
S_{p,1/6}\equiv a_p(f_{72.4.a.b})\pmod{p^3}.
\]
Li--Long--Tu \cite{MR4445932} identify the quadratic twists
\[
a_p(f_{36.4.a.a})
=
\left(\frac{-3}{p}\right)a_p(f_{12.4.a.a}),
\qquad
a_p(f_{72.4.a.b})
=
\left(\frac{-3}{p}\right)a_p(f_{24.4.a.a}).
\]
Consequently, for $p\geq7$, Theorem~\ref{thm:main} reduces in both cases to
the hypergeometric congruence
\begin{equation}\label{eq:target-bridge}
	H_{p,\rho}
	\equiv
	\left(\frac{-3}{p}\right)S_{p,\rho}
	\pmod{p^3},
	\qquad
	\rho\in\left\{\frac13,\frac16\right\}.
\end{equation}

It therefore remains to establish \eqref{eq:target-bridge}.  The proofs of
the two cases follow the same global strategy.  We first deform the truncated
${}_7F_6(1)$ series to a terminating specialization.  At the terminating
point, Mishev's symmetry together with classical Whipple and Sears
transformations reduces the problem to a shifted rank-four hypergeometric
sum.  The remaining finite factors are simplified modulo $p^3$, and a
normalized $p$-adic deformation transports the shifted rank-four sum back
to $S_{p,\rho}$.  The two cases share this global structure, while the
initial valuation argument and certain endpoint $p$-adic gamma calculations
are case-dependent.
\section{The Cubic Deformation}
\label{sec:cubic-deformation}
Throughout the remainder of the proof, the prime $p\geq7$ is fixed.
We therefore suppress the dependence on $p$ from the notation for
auxiliary quantities introduced below, while retaining the notation
$H_{p,\rho}$ and $S_{p,\rho}$ for the original truncated series.  Put
\[
N:=\frac{p-1}{2},
\qquad
x:=\frac p2=N+\frac12,
\]
and let $\omega$ be a primitive cube root of unity, so that
\[
1+\omega+\omega^2=0.
\]
When $p\equiv1\pmod3$, we may take $\omega\in\mathbb Q_p$; when
$p\equiv2\pmod3$, we work in the unramified quadratic extension
$\mathbb Q_p(\omega)$.  In either case, $v_p$ denotes the extension of
the $p$-adic valuation normalized by $v_p(p)=1$.

The purpose of this section is to replace the original truncation
$H_{p,\rho}$ by a terminating very-well-poised ${}_7F_6(1)$ series
without changing its value modulo $p^3$.  Define
\begin{equation}\label{eq:H-rho-t}
	H_\rho(t)
	:=
	{}_7F_6\!\left[
	\begin{matrix}
		\frac12,\frac54,
		\frac12+t,\frac12+\omega t,\frac12+\omega^2t,
		\rho,1-\rho\\[1mm]
		\frac14,
		1-t,1-\omega t,1-\omega^2t,
		\frac32-\rho,\frac12+\rho
	\end{matrix};1
	\right].
\end{equation}
The deformation is adapted to the cyclic symmetry of the three repeated
parameters.  Indeed, replacing $t$ by $\omega t$ cyclically
permutes the three shifted upper parameters and the corresponding three
denominator parameters.  Hence
\[
H_\rho(\omega t)=H_\rho(t).
\]
It follows that the expansion of $H_\rho(t)$ about $t=0$ contains only
powers of $t^3$.  This cubic symmetry is reflected algebraically in the
identity
\begin{equation}\label{eq:cubic-identity}
	(y+t)(y+\omega t)(y+\omega^2t)=y^3+t^3.
\end{equation}
Thus the first nontrivial deformation term is cubic, making the
specialization $t=-p/2$ naturally suited to congruences modulo $p^3$.

At the same time, since $x=p/2=N+\frac12$, the specialization
$t=-x=-p/2$ gives
\[
\frac12+t=-N.
\]
Thus one upper parameter becomes a nonpositive integer, so
$H_\rho(-p/2)$ terminates at $k=N$.

Let $b_{\rho,k}$ denote the $k$th summand of $H_\rho(0)$.  Since
\[
\frac{(\frac54)_k}{(\frac14)_k}=4k+1,
\]
we have
\begin{equation}\label{eq:b-rho-k}
	b_{\rho,k}
	=
	(4k+1)
	\frac{(\frac12)_k^4(\rho)_k(1-\rho)_k}
	{(k!)^4(\frac32-\rho)_k(\frac12+\rho)_k}.
\end{equation}
In particular,
\begin{equation}\label{eq:H-as-sum}
	H_{p,\rho}=\sum_{k=0}^{p-1} b_{\rho,k}.
\end{equation}
For $k\leq N$, the ratio of the $k$th summand of $H_\rho(t)$ to
$b_{\rho,k}$ is independent of $\rho$.  Using
\eqref{eq:cubic-identity}, we obtain
\begin{equation}\label{eq:Rk}
	R_k(t)
	=
	\prod_{j=0}^{k-1}
	\frac{(j+\frac12)^3+t^3}{(j+\frac12)^3}
	\frac{(j+1)^3}{(j+1)^3-t^3}.
\end{equation}
For $0\leq j<N$, both $j+\frac12$ and $j+1$ are $p$-adic units.
Hence, at $t=-p/2$,
\begin{equation}\label{eq:Rk-expansion}
	R_k(-p/2)
	=
	1-\frac{p^3}{8}C_k+O(p^6),
	\qquad
	C_k:=
	\sum_{j=0}^{k-1}
	\left(
	\frac{1}{(j+\frac12)^3}
	+
	\frac{1}{(j+1)^3}
	\right)
	\in\Zp.
\end{equation}
Thus the deformation error is already of order $p^3$.  The remaining
issue is whether the coefficient $b_{\rho,k}$ can lower its
$p$-adic order, and this is precisely where the two cases differ. We begin with the case $\rho=\frac16$, where the deformation
congruence holds termwise.  The case $\rho=\frac13$ follows the same
pattern but contains an exceptional block requiring an additional
cancellation argument.

\subsection{The case $\rho=\frac16$: no exceptional block}

\leavevmode\par
By \eqref{eq:b-rho-k},
\[
b_{1/6,k}
=
(4k+1)
\frac{
	(\frac12)_k^4(\frac16)_k(\frac56)_k
}{
	(k!)^4(\frac43)_k(\frac23)_k
}.
\]
Thus the rational Pochhammer symbols relevant to the valuation analysis are
$(\frac12)_k$, $(\frac16)_k$, $(\frac56)_k$ in the numerator and
$(\frac43)_k$, $(\frac23)_k$ in the denominator.

\begin{lemma}\label{lem:116-initial}
	For $0\leq k\leq N$, one has
	\[
	v_p(b_{1/6,k})\geq0.
	\]
\end{lemma}

\begin{proof}
	For $k<p$, the factor $(k!)^4$ is a $p$-adic unit.  Each rational
	Pochhammer symbol displayed above contains at most one factor divisible
	by $p$, since the corresponding linear congruence has a unique solution
	modulo $p$.  After clearing the denominators $2$, $3$, and $6$, every
	relevant positive linear factor is at most $6p-5$.  Since
	\[
	6p-5<p^2
	\qquad (p\geq7),
	\]
	each such $p$-divisible factor has valuation exactly one.
	
	The first crossing thresholds are
	\[
	\begin{array}{c|ccccc}
		&(\frac12)_k&(\frac16)_k&(\frac56)_k
		&(\frac43)_k&(\frac23)_k\\ \hline
		p=6m+1&3m+1&m+1&5m+1&2m&4m+1\\
		p=6m+5&3m+3&5m+5&m+1&4m+3&2m+2
	\end{array}
	\]
	
	If $p=6m+1$, then $N=3m$.  The denominator factor
	$(\frac43)_k$ first acquires a $p$-factor at $k=2m$, but
	$(\frac16)_k$ has already acquired one at $k=m+1$; moreover,
	$(\frac23)_k$ does not cross before $N$.
	
	If $p=6m+5$, then $N=3m+2$.  The denominator factor
	$(\frac23)_k$ first acquires a $p$-factor at $k=2m+2$, but
	$(\frac56)_k$ has already acquired one at $k=m+1$; moreover,
	$(\frac43)_k$ does not cross before $N$.
	
	The remaining factor $4k+1$ lies in the numerator and therefore cannot
	decrease the valuation.  Thus the denominator never contributes greater
	$p$-adic valuation than the numerator, and hence
	\[
	v_p(b_{1/6,k})\geq0.
	\]
\end{proof}

It follows from \eqref{eq:Rk-expansion} that
\begin{equation}\label{eq:116-initial-termwise}
	b_{1/6,k}R_k(-p/2)
	\equiv
	b_{1/6,k}
	\pmod{p^3},
	\qquad
	0\leq k\leq N.
\end{equation}

It remains to remove the upper half of the original truncation.
For $N<k<p$, the four copies of $(\frac12)_k$ contribute valuation
$4$, and at least one of $(\frac16)_k$ and $(\frac56)_k$ contributes
one further $p$-factor.  The denominator factors $(\frac43)_k$ and
$(\frac23)_k$ contribute at most one $p$-factor each.  Hence
\begin{equation}\label{eq:116-tail-valuation}
	v_p(b_{1/6,k})\geq 4+1-2=3,
	\qquad
	N<k<p.
\end{equation}
Combining \eqref{eq:116-initial-termwise} with
\eqref{eq:116-tail-valuation} and using
\eqref{eq:H-as-sum}, we obtain
\begin{equation}\label{eq:initial-bridge-116}
	H_{p,1/6}
	\equiv
	H_{1/6}(-p/2)
	\pmod{p^3}.
\end{equation}

\subsection{The case $\rho=\frac13$: cancellation of the exceptional block}
\leavevmode\par

By \eqref{eq:b-rho-k},
\[
b_{1/3,k}
=
(4k+1)
\frac{
	(\frac12)_k^4(\frac13)_k(\frac23)_k
}{
	(k!)^4(\frac76)_k(\frac56)_k
}.
\]
Thus the rational Pochhammer symbols relevant to the valuation analysis are
$(\frac12)_k$, $(\frac13)_k$, $(\frac23)_k$ in the numerator and
$(\frac76)_k$, $(\frac56)_k$ in the denominator. The situation is slightly subtler than in the case $\rho=\frac16$.
For $k<p$, the factor $(k!)^4$ is a $p$-adic unit, while
$b_{1/3,k}$ can have valuation $-1$ on a short interval.
Consequently, \eqref{eq:Rk-expansion} does not yield a termwise
congruence modulo $p^3$ throughout $0\leq k\leq N$.
We shall show that the resulting order-$p^2$ errors cancel in pairs
under the involution
\[
k\longmapsto N-k.
\]

Write $p=6m+1$ or $p=6m+5$, and define
\[
I_p:=
\begin{cases}
	\{m,m+1,\ldots,2m\},&p=6m+1,\\
	\{m+1,m+2,\ldots,2m+1\},&p=6m+5.
\end{cases}
\]
The first crossing thresholds are
\[
\begin{array}{c|ccccc}
	&(\frac12)_k&(\frac13)_k&(\frac23)_k
	&(\frac76)_k&(\frac56)_k\\ \hline
	p=6m+1&3m+1&2m+1&4m+1&m&5m+1\\
	p=6m+5&3m+3&4m+4&2m+2&5m+4&m+1
\end{array}
\]
and therefore
\begin{equation}\label{eq:113-initial-valuations}
	v_p(b_{1/3,k})\geq0
	\qquad
	(k\notin I_p,\ 0\leq k\leq N),
\end{equation}
whereas
\[
v_p(b_{1/3,k})\geq-1
\qquad
(k\in I_p).
\]
At the possible central index $k=(p-1)/4$, the factor
$4k+1=p$ raises the valuation by one.

For $k\in I_p$, define
\[
q_k:=p\,b_{1/3,k}\pmod p.
\]
\begin{lemma}\label{lem:C-symmetry}
	For $k\in I_p$,
	\[
	C_{N-k}\equiv C_k\pmod p.
	\]
\end{lemma}

\begin{proof}
	Put
	\[
	c_j:=
	\frac{1}{(j+\frac12)^3}
	+
	\frac{1}{(j+1)^3},
	\qquad
	C_k=\sum_{j=0}^{k-1}c_j.
	\]
	Since $N\equiv-\frac12\pmod p$,
	\[
	N-\frac12-j\equiv-(j+1)\pmod p,
	\qquad
	N-j\equiv-\left(j+\frac12\right)\pmod p,
	\]
	and hence
	\[
	c_{N-1-j}\equiv-c_j\pmod p.
	\]
	It follows that $C_N\equiv0\pmod p$.  Moreover,
	\[
	C_N-C_{N-k}
	=
	\sum_{j=N-k}^{N-1}c_j
	\equiv
	-\sum_{j=0}^{k-1}c_j
	=
	-C_k
	\pmod p,
	\]
	which gives the claim.
\end{proof}

\begin{lemma}\label{lem:q-antisymmetry}
	For $k\in I_p$,
	\[
	q_{N-k}\equiv-q_k\pmod p.
	\]
\end{lemma}

\begin{proof}
	The ratio of consecutive terms is
	\[
	r_k:=
	\frac{b_{1/3,k+1}}{b_{1/3,k}}
	=
	\frac{4k+5}{4k+1}
	\frac{(k+\frac12)^4(k+\frac13)(k+\frac23)}
	{(k+1)^4(k+\frac56)(k+\frac76)}.
	\]
	As a rational function of $z$,
	\[
	r(z)\,r\left(-z-\frac32\right)=1.
	\]
	Since
	\[
	N-k-1\equiv-k-\frac32\pmod p,
	\]
	we obtain
	\[
	r_kr_{N-k-1}\equiv1\pmod p
	\]
	away from the possible central zero or pole.  Thus
	$q_{N-k}/q_k$ is constant as one moves outward through paired indices.
	
	If $N$ is odd, let $c=(N-1)/2$.  Then $N-c=c+1$ and
	$c\equiv-3/4\pmod p$, so direct substitution gives
	\[
	r_c\equiv-1\pmod p.
	\]
	If $N$ is even, let $c=N/2=(p-1)/4$.  Then $q_c=0$, and after
	cancelling the central zero and pole one obtains
	\[
	r_{c-1}r_c\equiv-1\pmod p.
	\]
	In either case the sign $-1$ propagates through all paired indices,
	which proves the result.
\end{proof}

Since $I_p$ is stable under $k\mapsto N-k$,
Lemmas~\ref{lem:C-symmetry} and \ref{lem:q-antisymmetry} imply
\[
\sum_{k\in I_p}q_kC_k\equiv0\pmod p.
\]
Hence the contribution of the exceptional block to the deformation
error satisfies
\begin{equation}\label{eq:113-exceptional-cancellation}
	\sum_{k\in I_p}
	b_{1/3,k}\bigl(R_k(-p/2)-1\bigr)
	\equiv
	-\frac{p^2}{8}
	\sum_{k\in I_p}(p\,b_{1/3,k})C_k
	\equiv0
	\pmod{p^3}.
\end{equation}

For $k\notin I_p$, \eqref{eq:113-initial-valuations} together with
\eqref{eq:Rk-expansion} gives the termwise congruence
\[
b_{1/3,k}R_k(-p/2)
\equiv
b_{1/3,k}
\pmod{p^3}.
\]
Combining this with \eqref{eq:113-exceptional-cancellation}, we obtain
\begin{equation}\label{eq:113-initial-lower-half}
	\sum_{k=0}^{N}
	b_{1/3,k}R_k(-p/2)
	\equiv
	\sum_{k=0}^{N}b_{1/3,k}
	\pmod{p^3}.
\end{equation}

Finally, for $N<k<p$, the four copies of $(\frac12)_k$ contribute
valuation $4$.  Inspection of the remaining crossings in the table above
shows that the numerator and denominator contributions have net valuation
at least zero.  Therefore
\begin{equation}\label{eq:113-tail-valuation}
	v_p(b_{1/3,k})\geq4,
	\qquad
	N<k<p.
\end{equation}

Using \eqref{eq:H-as-sum}, \eqref{eq:113-initial-lower-half}, and
\eqref{eq:113-tail-valuation}, we conclude that
\begin{equation}\label{eq:initial-bridge-113}
	H_{p,1/3}
	\equiv
	H_{1/3}(-p/2)
	\pmod{p^3}.
\end{equation}
Thus the two different valuation arguments lead to the same
terminating specialization.
\begin{proposition}\label{prop:initial-bridge}
	For $\rho\in\{\frac13,\frac16\}$, one has
	\begin{equation}\label{eq:initial-bridge}
		H_{p,\rho}\equiv H_\rho(-p/2)\pmod{p^3}.
	\end{equation}
\end{proposition}

\section{Mishev's $W(D_5)$ Reflection}

The cubic deformation of the previous section produces a terminating
very-well-poised ${}_7F_6(1)$ series at $t=-p/2$.  Our next goal is to
move this series to a parameter configuration in which Whipple's
transformation can be applied effectively.  We do this by embedding the
deformation into Mishev's completed $L$-function and applying a
fundamental reflection in its $W(D_5)$ symmetry group.

On the Saalsch\"utz hyperplane
\begin{equation}\label{eq:saalschutz-hyperplane}
	e+f+g-a-b-c-d=1,
\end{equation}
Mishev's completed function admits the very-well-poised representation
\cite{MR2834915}
\begin{align}
	L(a,b,c,d;e;f,g)
	&=
	\frac{\Gamma(1+d+g-e)}
	{\pi\Gamma(g)\Gamma(1+g-e)\Gamma(f-d)
		\Gamma(1+a+d-e)\Gamma(1+b+d-e)\Gamma(1+c+d-e)}
	\notag\\
	&\qquad\times
	{}_7F_6\!\left[
	\begin{matrix}
		\alpha,\ 1+\frac{\alpha}{2},\ g-a,\ g-b,\ g-c,\ d,\ 1+d-e\\[1mm]
		\frac{\alpha}{2},\ 1+a+d-e,\ 1+b+d-e,\ 1+c+d-e,\ 1+g-e,\ g
	\end{matrix};1
	\right],
	\label{eq:Mishev-L-vwp}
\end{align}
where
\[
\alpha=d+g-e.
\]

The function $L$ is invariant under the fundamental reflection
\cite{MR2834915}
\begin{equation}\label{eq:Mishev-reflection}
	L(a,b,c,d;e;f,g)
	=
	L\bigl(\mathcal A(a,b,c,d;e;f,g)\bigr),
\end{equation}
where
\begin{equation}\label{eq:Mishev-A}
	\mathcal A(a,b,c,d;e;f,g)
	=
	\bigl(
	a,b,g-c,g-d;
	1+a+b-f;
	1+a+b-e,g
	\bigr).
\end{equation}

For $\rho\in\{\frac13,\frac16\}$, consider the path
\begin{equation}\label{eq:Qrho}
	Q_\rho(t)
	=
	\left(
	\rho-t,\rho-\omega t,\rho-\omega^2t,\rho;
	2\rho;
	\frac12+\rho,\frac12+\rho
	\right).
\end{equation}
Since $1+\omega+\omega^2=0$, the path $Q_\rho(t)$ satisfies the
Saalsch\"utz condition \eqref{eq:saalschutz-hyperplane}.

Substituting \eqref{eq:Qrho} into the representation
\eqref{eq:Mishev-L-vwp} gives
\[
L(Q_\rho(t))
=
M_{Q,\rho}(t)H_\rho(t),
\]
where
\[
M_{Q,\rho}(t)
=
\frac{\Gamma(\frac32)}
{\pi
	\Gamma(\frac12+\rho)
	\Gamma(\frac32-\rho)
	\Gamma(\frac12)
	\Gamma(1-t)
	\Gamma(1-\omega t)
	\Gamma(1-\omega^2t)}.
\]

Applying the reflection \eqref{eq:Mishev-reflection} to
\eqref{eq:Qrho}, we obtain
\[
\mathcal A Q_\rho(t)
=
\left(
\rho-t,\rho-\omega t,\frac12+\omega^2t,\frac12;
\frac12+\rho+\omega^2t;
1+\omega^2t,\frac12+\rho
\right).
\]
A second application of \eqref{eq:Mishev-L-vwp} gives
\[
L(\mathcal A Q_\rho(t))
=
M_{\mathcal A,\rho}(t)W_\rho(t),
\]
where
\[
M_{\mathcal A,\rho}(t)
=
\frac{\Gamma(\frac32-\omega^2t)}
{\pi
	\Gamma(\frac12+\rho)
	\Gamma(1-\omega^2t)
	\Gamma(\frac12+\omega^2t)
	\Gamma(1+\omega t)
	\Gamma(1+t)
	\Gamma(\frac32-\rho)},
\]
and
\[
W_\rho(t)
=
{}_7F_6\!\left[
\begin{matrix}
	\frac12-\omega^2t,\
	\frac54-\frac{\omega^2t}{2},\
	\frac12+t,\
	\frac12+\omega t,\
	\rho-\omega^2t,\
	\frac12,\
	1-\rho-\omega^2t
	\\[1mm]
	\frac14-\frac{\omega^2t}{2},\
	1+\omega t,\
	1+t,\
	\frac32-\rho,\
	1-\omega^2t,\
	\frac12+\rho
\end{matrix};1
\right].
\]

By the reflection identity \eqref{eq:Mishev-reflection},
\[
L(Q_\rho(t))
=
L(\mathcal A Q_\rho(t)).
\]
Comparing the two representations above, all factors involving $\rho$
cancel from the quotient of the gamma prefactors, and
$\Gamma(\frac12)/\Gamma(\frac32)=2$.  Hence
\begin{equation}\label{eq:Mishev-universal-quotient}
	H_\rho(t)
	=
	2
	\frac{
		\Gamma(\frac32-\omega^2t)
		\Gamma(1-t)
		\Gamma(1-\omega t)
	}{
		\Gamma(\frac12+\omega^2t)
		\Gamma(1+\omega t)
		\Gamma(1+t)
	}
	W_\rho(t).
\end{equation}

The crucial point is that the prefactor in
\eqref{eq:Mishev-universal-quotient} is independent of $\rho$.
Thus, after the case-dependent valuation analysis of Section~\ref{sec:cubic-deformation}, the
transformation mechanism is common to both supercongruences.  At
$t=-p/2$, the parameter $\frac12+t=-N$ in $W_\rho(t)$ again forces
termination, and the resulting series is in the form needed for the
Whipple transformation of the next section.

\section{Terminating Whipple and the First Lower-Rank Form}

We now specialize the transformed series $W_\rho(t)$ at $t=-p/2$.
The resulting ${}_7F_6(1)$ terminates because one of its upper
parameters is $-N$.  The purpose of this section is to apply Whipple's
terminating transformation and thereby reduce the problem from a
${}_7F_6(1)$ series to a rank-four hypergeometric series.

We use Whipple's terminating transformation in the standard form
\cite[Theorem~3.4.4]{MR1688958}
\begin{align}
	&{}_7F_6\!\left[
	\begin{matrix}
		a,1+\frac a2,c,d,e,f,g\\[1mm]
		\frac a2,1+a-c,1+a-d,1+a-e,1+a-f,1+a-g
	\end{matrix};1
	\right]
	\notag\\
	&\qquad=
	\mathcal P\,
	{}_4F_3\!\left[
	\begin{matrix}
		1+a-c-d,e,f,g\\
		e+f+g-a,1+a-c,1+a-d
	\end{matrix};1
	\right],
	\label{eq:Whipple}
\end{align}
where
\[
\mathcal P
=
\frac{
	\Gamma(1+a-e)\Gamma(1+a-f)\Gamma(1+a-g)
	\Gamma(1+a-e-f-g)}
{
	\Gamma(1+a)\Gamma(1+a-f-g)
	\Gamma(1+a-e-f)\Gamma(1+a-e-g)
}.
\]

In $W_\rho(-x)$, take
\[
a=\frac12+\om^2x,\qquad
c=\frac12-\om x,\qquad
d=\frac12,
\]
\[
e=-N,\qquad
f=\rho+\om^2x,\qquad
g=1-\rho+\om^2x.
\]
The transformed parameters simplify to
\[
1+a-c-d=-N,
\]
\[
e+f+g-a=1-(1-\om^2)x,
\]
and
\[
1+a-c=1-x,
\qquad
1+a-d=1+\om^2x.
\]
Therefore Whipple's transformation produces the terminating rank-four
series
\begin{equation}\label{eq:Vrho}
	V_\rho
	:=
	{}_4F_3\!\left[
	\begin{matrix}
		-N,-N,\rho+\om^2x,1-\rho+\om^2x\\
		1-(1-\om^2)x,1-x,1+\om^2x
	\end{matrix};1
	\right]_N.
\end{equation}

Since $e=-N$, the gamma quotient in \eqref{eq:Whipple} reduces to
\begin{equation}\label{eq:Whipple-prefactor-rho}
	\mathcal P_\rho
	=
	\frac{
		(\frac32+\omega^2x)_N
		(\frac12-\omega^2x)_N
	}{
		(\frac32-\rho)_N
		(\frac12+\rho)_N
	}.
\end{equation}
It is convenient to isolate the $\rho$-dependent finite factor
\begin{equation}\label{eq:Rrho}
	R_\rho
	:=
	p\,
	\frac{(\frac12)_N^2}
	{(\frac32-\rho)_N(\frac12+\rho)_N}.
\end{equation}

The remaining ordinary-gamma simplification is common to both values
of $\rho$.  Since
\[
x=N+\frac12
\qquad\text{and}\qquad
1+\om+\om^2=0,
\]
we have
\[
\frac12-\om^2x=N+1+\om x,
\qquad
\frac32+\om^2x=1-N-\om x.
\]
Moreover,
\begin{equation}\label{eq:ordinary-gamma-universal}
	\frac{\Gamma(1+x)}{\Gamma(1-x)}
	=
	(-1)^N\frac p2
	\left(\frac12\right)_N^2.
\end{equation}

Substituting these identities into
\eqref{eq:Mishev-universal-quotient} and combining the resulting
gamma factors with \eqref{eq:Whipple-prefactor-rho} gives the exact
factorization
\begin{equation}\label{eq:Whipple-exact-factorization}
	H_\rho(-p/2)=\Pi_\rho V_\rho,
\end{equation}
where
\begin{equation}\label{eq:Pi-rho}
	\Pi_\rho
	=
	(-1)^N R_\rho
	\frac{(N+1+\omega x)_N}{(1+\omega x)_N}.
\end{equation}

\section{The Terminating Sears Transformation}

The Whipple transformation of the previous section has reduced the
problem to the terminating rank-four series $V_\rho$.  We now apply a
terminating Sears transformation.  Its role is to replace the two
$\rho$-dependent parameters involving $\om^2x$ by the common shifts
$\rho-x$ and $1-\rho-x$.  After this step, the remaining dependence on
$\om$ occurs only in the denominator and in a finite prefactor.

We use the terminating Saalsch\"utzian Sears transformation \cite[Corollary~4.2 together with (4.7)]{MR2765604}
\begin{equation}\label{eq:Sears}
	{}_4F_3\!\left[
	\begin{matrix}
		-n,a,b,c\\
		d,e,f
	\end{matrix};1
	\right]
	=
	\frac{(e-a)_n(f-a)_n}{(e)_n(f)_n}
	{}_4F_3\!\left[
	\begin{matrix}
		-n,a,d-b,d-c\\
		d,a-e-n+1,a-f-n+1
	\end{matrix};1
	\right].
\end{equation}

Apply \eqref{eq:Sears} to $V_\rho$ with
\[
n=N,\qquad
a=-N,\qquad
b=\rho+\om^2x,\qquad
c=1-\rho+\om^2x,
\]
and
\[
d=1-(1-\om^2)x,\qquad
e=1-x,\qquad
f=1+\om^2x.
\]
The transformed parameters simplify to
\[
d-b=1-\rho-x,
\qquad
d-c=\rho-x,
\]
and
\[
a-e-N+1=1-x,
\qquad
a-f-N+1=1+(\om-1)x.
\]
Since $-N=\frac12-x$, the resulting series is
\begin{equation}\label{eq:Jrho}
	J_\rho
	:=
	{}_4F_3\!\left[
	\begin{matrix}
		\frac12-x,\frac12-x,\rho-x,1-\rho-x\\
		1-x,1+(\om-1)x,1+(\om^2-1)x
	\end{matrix};1
	\right]_N.
\end{equation}

The Sears prefactor is independent of $\rho$.  Indeed,
\[
\frac{(e-a)_N(f-a)_N}{(e)_N(f)_N}
=
\frac{
	(\frac12)_N(\frac12-\om x)_N
}{
	(1-x)_N(1+\om^2x)_N
}.
\]
Using the reversal identity
\[
(a)_N=(-1)^N(1-a-N)_N,
\]
we obtain
\[
(1-x)_N
=
(-1)^N\left(\frac12\right)_N
\]
and
\[
\left(\frac12-\om x\right)_N
=
(-1)^N(1+(\om-1)x)_N.
\]
Hence
\begin{equation}\label{eq:Sears-prefactor}
	V_\rho
	=
	\frac{(1+(\om-1)x)_N}
	{(1+\om^2x)_N}
	J_\rho.
\end{equation}

A further application of the reversal identity gives
\begin{equation}\label{eq:reversal-K}
	(-1)^N(N+1+\om x)_N
	=
	(1+(\om^2-1)x)_N.
\end{equation}
Combining \eqref{eq:Whipple-exact-factorization},
\eqref{eq:Pi-rho}, \eqref{eq:Sears-prefactor}, and
\eqref{eq:reversal-K}, we obtain the exact factorization
\begin{equation}\label{eq:exact-factorization}
	H_\rho(-p/2)
	=
	R_\rho KJ_\rho,
\end{equation}
where the factor
\begin{equation}\label{eq:K}
	K
	:=
	\frac{
		(1+(\om-1)x)_N
		(1+(\om^2-1)x)_N
	}{
		(1+\om x)_N
		(1+\om^2x)_N
	}
\end{equation}
is independent of $\rho$.

\section{Cubic Denominator Collapse}

The numerator of $J_\rho$ already has the desired common $-x$ shift.
The only remaining dependence on $\omega$ occurs in two denominator
parameters.  We now remove this dependence modulo $p^3$ using the
cubic identity \eqref{eq:cubic-identity} from Section~\ref{sec:cubic-deformation}.

Define
\begin{equation}\label{eq:U-rho}
	U_\rho
	:=
	\sum_{k=0}^{N}
	\frac{
		(\frac12-x)_k^2(\rho-x)_k(1-\rho-x)_k
	}{
		(1-x)_k^4
	}.
\end{equation}
We compare the $k$th summands of $J_\rho$ and $U_\rho$.
Their numerators are identical, so it remains only to compare the
denominators.  We claim that
\begin{equation}\label{eq:cubic-denominator}
	k!\,
	(1+(\omega-1)x)_k
	(1+(\omega^2-1)x)_k
	\equiv
	(1-x)_k^3
	\pmod{p^3}.
\end{equation}

For $0\leq j<k\leq N$, put
\[
y:=j+1-x.
\]
Then $y$ is a $p$-adic unit, and the three factors on the left-hand
side of \eqref{eq:cubic-denominator} contributed at index $j$ are
\[
y+x,\qquad y+\omega x,\qquad y+\omega^2x.
\]
By \eqref{eq:cubic-identity}, their product is $y^3+x^3$.
Since $x=p/2$, we have $x^3=p^3/8$, and hence
\[
(y+x)(y+\omega x)(y+\omega^2x)
\equiv
y^3
\pmod{p^3}.
\]
Multiplying over $j=0,\ldots,k-1$ proves
\eqref{eq:cubic-denominator}.

For $k\leq N$, all denominator factors are $p$-adic units, so
\eqref{eq:cubic-denominator} may be inverted termwise.  We obtain the
following common reduction.

\begin{lemma}\label{lem:cubic-collapse}
	For $\rho\in\{\frac13,\frac16\}$, one has
	\[
	J_\rho\equiv U_\rho\pmod{p^3}.
	\]
\end{lemma}

\section{The Finite Prefactor and Morley's Congruence}

After the cubic denominator collapse, the remaining universal finite
factor is $K$.  We now simplify it modulo $p^3$ and show that it is
congruent to the square of the central binomial factor
\[
c_N:=\frac{(\frac12)_N}{N!}.
\]

Pairing the conjugate factors in \eqref{eq:K}, for $a=1,\ldots,N$ we have
\[
(a+(\om-1)x)(a+(\om^2-1)x)
=
a^2-3ax+3x^2,
\]
and
\[
(a+\om x)(a+\om^2x)
=
a^2-ax+x^2.
\]
Therefore
\begin{equation}\label{eq:K-product}
	K
	=
	\prod_{a=1}^{N}
	\frac{a^2-3ax+3x^2}{a^2-ax+x^2}.
\end{equation}

For each $a=1,\ldots,N$, put
\[
u:=\frac{x}{a}=\frac{p}{2a}.
\]
Since $a$ is a $p$-adic unit,
\[
\frac{1-3u+3u^2}{1-u+u^2}
=
1-2u+\frac{2u^3}{1-u+u^2}
\equiv
1-\frac pa
\pmod{p^3}.
\]
Multiplying over $a=1,\ldots,N$ gives
\begin{equation}\label{eq:K-binomial}
	K
	\equiv
	\prod_{a=1}^{N}\left(1-\frac pa\right)
	=
	(-1)^N\binom{p-1}{N}
	\pmod{p^3}.
\end{equation}

Morley's congruence \cite{MR1502188} states that
\begin{equation}\label{eq:Morley}
	(-1)^N\binom{p-1}{N}
	\equiv
	4^{p-1}
	\pmod{p^3}.
\end{equation}
Set
\begin{equation}\label{eq:cN}
	c_N
	:=
	\frac{(\frac12)_N}{N!}
	=
	\frac{1}{4^N}\binom{p-1}{N}.
\end{equation}
Since $2N=p-1$, equations \eqref{eq:K-binomial} and \eqref{eq:Morley}
imply
\[
c_N^2
=
\frac{1}{4^{p-1}}\binom{p-1}{N}^2
\equiv
4^{p-1}
\equiv
K
\pmod{p^3}.
\]
Thus
\begin{equation}\label{eq:K-cN}
	K\equiv c_N^2\pmod{p^3}.
\end{equation}

Combining \eqref{eq:exact-factorization},
Lemma~\ref{lem:cubic-collapse}, and \eqref{eq:K-cN}, we obtain
\begin{equation}\label{eq:before-common-shift}
	H_\rho(-p/2)
	\equiv
	R_\rho c_N^2 U_\rho
	\pmod{p^3}.
\end{equation}

\section{The Normalized $p$-Adic Common Shift}

We now connect the shifted sum $U_\rho$ back to the original rank-four
truncation $S_{p,\rho}$.  The main perturbation calculation used in
this section is due to Long--Tu--Yui--Zudilin
\cite[Lemma~5 and its proof]{MR4330088}.  We apply their argument to
the two rank-four data occurring here and package it in a normalized
$p$-adic family.  The only additional step is to extend their
nonnegative integral perturbation parameter to $s\in\Zp$ by
$p$-adic continuity, allowing the specialization $s=-\frac12$.

Let
\[
\alpha_\rho
:=
\left\{\frac12,\frac12,\rho,1-\rho\right\}
=
\{r_1,r_2,r_3,r_4\}.
\]
For $\rho=\frac13$ and $\rho=\frac16$, these are among the rigid
rank-four hypergeometric data treated in \cite{MR4330088}.

Let $\Gamp$ denote Morita's $p$-adic gamma function
\cite{MR424762}.  For $s\in\Zp$, define
\begin{equation}\label{eq:G-rho-s}
	\mathcal G_\rho(s)
	:=
	\frac{\prod_{j=1}^4\Gamp(r_j+sp)}
	{\Gamp(1+sp)^4}
	\sum_{k=0}^{p-1}
	\frac{\prod_{j=1}^4(r_j+sp)_k}
	{(1+sp)_k^4}.
\end{equation}

\begin{proposition}[Normalized common-shift invariance]
	\label{prop:common-shift}
	For $\rho\in\{\frac13,\frac16\}$ and every $s\in\Zp$,
	\[
	\mathcal G_\rho(s)\equiv\mathcal G_\rho(0)\pmod{p^3}.
	\]
\end{proposition}

\begin{proof}
	The main perturbation calculation is precisely the one carried out by
	Long--Tu--Yui--Zudilin \cite[Lemma~5 and its proof]{MR4330088}.
	We record only the specialization needed here.
	
With $\alpha_\rho=\{r_1,r_2,r_3,r_4\}$ as above, set
	\[
	A_k:=\frac{\prod_{j=1}^4(r_j)_k}{(k!)^4},
	\qquad
	\gamma_\rho:=\prod_{j=1}^4\Gamp(r_j).
	\]
	
	For each $r_j$, let $[-r_j]_0\in\{0,1,\ldots,p-1\}$ denote the
	least nonnegative residue of $-r_j$ modulo $p$.  Then
	$r_j+[-r_j]_0$ is divisible by $p$ in $\Zp$, so we may write
	\[
	r_j+[-r_j]_0=pr_j',
	\qquad r_j'\in\Zp.
	\]
	
	Define
	\[
	\Lambda_{\rho,k}(s)
	:=
	\prod_{j=1}^4
	\left(1+\frac{s}{r_j'}\right)^{\nu(k,[-r_j]_0)},
	\]
	where $\nu(k,a)=0$ for $k\leq a$ and $\nu(k,a)=1$ for $k>a$.
	
	Specializing the $p$-adic perturbation expansion of
	Long--Tu--Yui--Zudilin to the datum $\alpha_\rho$ gives
	\begin{equation}\label{eq:LTYZ-specialized}
		\mathcal G_\rho(s)-\mathcal G_\rho(0)
		\equiv
		\gamma_\rho
		\sum_{k=0}^{p-1}
		A_k
		\left[
		\Lambda_{\rho,k}(s)
		\bigl(1+J_1(k)sp+J_2(k)(sp)^2\bigr)-1
		\right]
		\pmod{p^3},
	\end{equation}
	where $J_1(k)$ and $J_2(k)$ are the coefficients arising from their
	local $p$-adic gamma expansion.
	
	For $s=b\in\mathbb Z_{\geq0}$, the sum on the right-hand side of
	\eqref{eq:LTYZ-specialized} is exactly the perturbation expression
	treated in \cite[Lemma~5]{MR4330088}; hence
	\[
	\mathcal G_\rho(b)\equiv\mathcal G_\rho(0)\pmod{p^3}.
	\]
	
	It remains only to extend the parameter from
	$b\in\mathbb Z_{\geq0}$ to $s\in\Zp$.
	The right-hand side of \eqref{eq:LTYZ-specialized} is a finite
	continuous $p$-adic function of $s$.  Since the nonnegative integers
	are dense in $\Zp$ and $p^3\Zp$ is closed, the same congruence holds
	for every $s\in\Zp$.
\end{proof}

\begin{remark}
	Only the specialization $s=-\frac12$ is needed below.  Proposition
	\ref{prop:common-shift} is stated in its natural $p$-adic form because
	the same family may have other useful terminating or resonant
	specializations.
\end{remark}

\section{Endpoint $p$-Adic Gamma Evaluations}

Proposition~\ref{prop:common-shift} identifies the normalized family
$\mathcal G_\rho(s)$ at $s=0$ and $s=-\frac12$ modulo $p^3$.
We now evaluate the gamma prefactors at these two endpoints.  The two
values of $\rho$ require different local calculations, but both lead to
the same rank-four congruence.

We recall the two standard identities for Morita's $p$-adic gamma
function that will be used repeatedly:
\begin{equation}\label{eq:Morita-functional}
	\frac{\Gamp(z+1)}{\Gamp(z)}
	=
	\begin{cases}
		-z,&z\in\Zp^\times,\\
		-1,&z\in p\Zp,
	\end{cases}
\end{equation}
together with the reflection formula
\begin{equation}\label{eq:Morita-reflection}
	\Gamp(z)\Gamp(1-z)=(-1)^{a_0(z)},
\end{equation}
where $a_0(z)\in\{1,\ldots,p\}$ is the unique integer congruent to
$z$ modulo $p$.

We shall also use the consequences
\begin{equation}\label{eq:universal-pgamma}
	\Gamp(-N)=\frac{1}{N!},
	\qquad
	\Gamp\left(\frac12-N\right)
	=
	\frac{\Gamp(\frac12)}{(\frac12)_N},
	\qquad
	\Gamp\left(\frac12\right)^4=1.
\end{equation}
The first two follow by iterating \eqref{eq:Morita-functional} without
crossing a multiple of $p$, and the last follows by squaring
\eqref{eq:Morita-reflection} at $z=\frac12$.

\subsection{The datum $\rho=\frac13$}

At $s=0$,
\[
\mathcal G_{1/3}(0)
=
\Gamp\left(\frac12\right)^2
\Gamp\left(\frac13\right)
\Gamp\left(\frac23\right)
S_{p,1/3}.
\]
Applying \eqref{eq:Morita-reflection} at $\frac12$ and $\frac13$
makes the sign explicit.  If $p=6m+1$, then
\[
a_0\left(\frac12\right)=3m+1,
\qquad
a_0\left(\frac13\right)=4m+1,
\]
so the product sign is $(-1)^m$.  If $p=6m+5$, then
\[
a_0\left(\frac12\right)=3m+3,
\qquad
a_0\left(\frac13\right)=2m+2,
\]
so the product sign is $(-1)^{m+1}$.  These are precisely the values
of $\left(\frac3p\right)$.  Hence
\begin{equation}\label{eq:G0-113}
	\mathcal G_{1/3}(0)
	=
	\left(\frac3p\right)S_{p,1/3}.
\end{equation}

At $s=-\frac12$, the two shifted parameters
$\frac12+sp$ become $-N$, so the sum terminates at $N$ and is exactly
$U_{1/3}$.  The corresponding gamma prefactor is
\[
\frac{
	\Gamp(-N)^2
	\Gamp(-N-\frac16)
	\Gamp(-N+\frac16)}
{\Gamp(\frac12-N)^4}.
\]
Put
\[
z_1:=-N-\frac16,
\qquad
z_2:=-N+\frac16.
\]
Then
\[
(z_1)_{N+1}
=
(-1)^{N+1}\frac16\left(\frac76\right)_N,
\qquad
(z_2)_N
=
(-1)^N\left(\frac56\right)_N.
\]

Exactly one of the shifts $z_1\to\frac56$ and $z_2\to\frac16$
crosses a multiple of $p$, and the crossed factor is $-p/6$.
Indeed, if $p=6m+1$, then
\[
z_1+2m=-\frac p6,
\]
whereas if $p=6m+5$, then
\[
z_2+(2m+1)=-\frac p6.
\]
Iterating \eqref{eq:Morita-functional} through the two strings gives
\[
\Gamp(z_1)\Gamp(z_2)
=
-p\,
\frac{
	\Gamp(\frac16)\Gamp(\frac56)}
{
	(\frac76)_N(\frac56)_N
}.
\]
By \eqref{eq:Morita-reflection},
\[
-\Gamp\left(\frac16\right)
\Gamp\left(\frac56\right)
=
(-1)^N.
\]
Indeed, if $p=6m+1$, then
$a_0(\frac16)=5m+1$ and $N=3m$, while if $p=6m+5$, then
$a_0(\frac16)=m+1$ and $N=3m+2$; in both cases
$a_0(\frac16)+1-N$ is even.

Together with \eqref{eq:universal-pgamma}, this yields
\begin{equation}\label{eq:half-113}
	\mathcal G_{1/3}\left(-\frac12\right)
	=
	(-1)^N R_{1/3}c_N^2U_{1/3}.
\end{equation}
Applying Proposition~\ref{prop:common-shift} and using
\[
(-1)^N=\left(\frac{-1}{p}\right)
\]
gives
\begin{equation}\label{eq:bridge-113}
	R_{1/3}c_N^2U_{1/3}
	\equiv
	\left(\frac{-3}{p}\right)S_{p,1/3}
	\pmod{p^3}.
\end{equation}

\subsection{The datum $\rho=\frac16$}

At $s=0$,
\[
\mathcal G_{1/6}(0)
=
\Gamp\left(\frac12\right)^2
\Gamp\left(\frac16\right)
\Gamp\left(\frac56\right)
S_{p,1/6}.
\]
We again make the reflection sign explicit.  If $p=6m+1$, then
\[
a_0\left(\frac12\right)=3m+1,
\qquad
a_0\left(\frac16\right)=5m+1,
\]
whose sum is even.  If $p=6m+5$, then
\[
a_0\left(\frac12\right)=3m+3,
\qquad
a_0\left(\frac16\right)=m+1,
\]
whose sum is again even.  Hence \eqref{eq:Morita-reflection} gives
\[
\Gamp\left(\frac12\right)^2
\Gamp\left(\frac16\right)
\Gamp\left(\frac56\right)
=
1,
\]
and therefore
\begin{equation}\label{eq:G0-116}
	\mathcal G_{1/6}(0)=S_{p,1/6}.
\end{equation}

At $s=-\frac12$, the shifted sum is $U_{1/6}$, and the gamma prefactor
is
\[
\frac{
	\Gamp(-N)^2
	\Gamp(-N-\frac13)
	\Gamp(-N+\frac13)}
{\Gamp(\frac12-N)^4}.
\]
Put
\[
z_1:=-N-\frac13,
\qquad
z_2:=-N+\frac13.
\]
Then
\begin{equation}\label{eq:z-products-116}
	(z_1)_{N+1}
	=
	(-1)^{N+1}\frac13\left(\frac43\right)_N,
	\qquad
	(z_2)_N
	=
	(-1)^N\left(\frac23\right)_N.
\end{equation}

Exactly one of the two strings crosses a multiple of $p$, and the
crossed factor is $-p/3$.  If $p=6m+1$, then $N=3m$ and
\[
z_1+m=-\frac p3;
\]
if $p=6m+5$, then $N=3m+2$ and
\[
z_2+m=-\frac p3.
\]

Before substituting \eqref{eq:z-products-116}, we record the direct
output of Morita's functional equation.  The two shifts have lengths
$N+1$ and $N$, respectively, so
\begin{equation}\label{eq:Morita-intermediate-116}
	\Gamp\left(\frac23\right)\Gamp\left(\frac13\right)
	=
	(-1)^{(N+1)+N}
	\Gamp(z_1)\Gamp(z_2)
	\frac{(z_1)_{N+1}(z_2)_N}{-p/3}.
\end{equation}
Since $(-1)^{2N+1}=-1$, rearranging gives
\[
\Gamp(z_1)\Gamp(z_2)
=
\frac p3
\frac{
	\Gamp(\frac13)\Gamp(\frac23)}
{(z_1)_{N+1}(z_2)_N}.
\]
Using \eqref{eq:z-products-116},
\[
(z_1)_{N+1}(z_2)_N
=
-\frac13
\left(\frac43\right)_N
\left(\frac23\right)_N,
\]
and therefore
\begin{equation}\label{eq:gamma-z-116}
	\Gamp(z_1)\Gamp(z_2)
	=
	-p\,
	\frac{
		\Gamp(\frac13)\Gamp(\frac23)}
	{
		(\frac43)_N(\frac23)_N
	}.
\end{equation}

By \eqref{eq:Morita-reflection},
\begin{equation}\label{eq:char-minus3}
	-\Gamp\left(\frac13\right)
	\Gamp\left(\frac23\right)
	=
	\left(\frac{-3}{p}\right).
\end{equation}
Indeed, the reflection exponent $a_0(\frac13)$ is $4m+1$ when
$p=6m+1$ and $2m+2$ when $p=6m+5$, giving respectively $+1$ and
$-1$ after the leading minus sign, exactly the values of
$\left(\frac{-3}{p}\right)$.

Combining \eqref{eq:universal-pgamma},
\eqref{eq:gamma-z-116}, and \eqref{eq:char-minus3}, we obtain
\begin{align}
	\mathcal G_{1/6}\left(-\frac12\right)
	&=
	\left(\frac{-3}{p}\right)
	p\,
	\frac{(\frac12)_N^4}
	{(N!)^2(\frac43)_N(\frac23)_N}
	U_{1/6}
	\notag\\
	&=
	\left(\frac{-3}{p}\right)
	R_{1/6}c_N^2U_{1/6}.
	\label{eq:half-116}
\end{align}
Proposition~\ref{prop:common-shift} and \eqref{eq:G0-116} therefore
give
\begin{equation}\label{eq:bridge-116}
	R_{1/6}c_N^2U_{1/6}
	\equiv
	\left(\frac{-3}{p}\right)S_{p,1/6}
	\pmod{p^3}.
\end{equation}

Thus the two endpoint calculations lead to the same rank-four bridge.

\begin{proposition}\label{prop:common-final-bridge}
	For $\rho\in\left\{\frac13,\frac16\right\}$,
	one has
	\[
	R_\rho c_N^2U_\rho
	\equiv
	\left(\frac{-3}{p}\right)S_{p,\rho}
	\pmod{p^3}.
	\]
\end{proposition}

\section{Completion of the Proof}

Combining Proposition~\ref{prop:initial-bridge},
\eqref{eq:before-common-shift}, and
Proposition~\ref{prop:common-final-bridge}, we obtain
\[
H_{p,\rho}
\equiv
\left(\frac{-3}{p}\right)S_{p,\rho}
\pmod{p^3},
\qquad
\rho\in\left\{\frac13,\frac16\right\}.
\]
This is precisely \eqref{eq:target-bridge}.  Theorem~\ref{thm:main}
therefore follows from the rank-four supercongruences and quadratic
twist relations recorded in Section~\ref{sec:main-results}. This proves the theorem for $p\geq7$; together with the direct verifications for $p=3$ and $p=5$, the proof is complete.

\section{Concluding Remarks}

From the Mishev reflection onward, the main transformation mechanism is
independent of the choice $\rho\in\{\frac13,\frac16\}$.
The case-dependent arithmetic is concentrated in the initial valuation
analysis, the finite factor $R_\rho$, and the endpoint
$p$-adic gamma evaluation.  By contrast, the Sears prefactor, the
universal factor $K$, the cubic denominator collapse, and the use of
Morley's congruence are common to both cases.  This separation between
case-dependent arithmetic and a common symmetry-and-transformation
mechanism is one of the main structural features of the proof.

The normalized common-shift invariance is also stronger than what is
needed here: only the specialization $s=-\frac12$ is used in the two
proofs.  It is natural to ask whether other rational points in the same
$p$-adic family give useful terminating or resonant specializations.
Together with the $W(D_5)$ symmetry and the classical Whipple--Sears
transformations, such specializations may provide a systematic way to
search for further hypergeometric supercongruences.
\section*{Acknowledgements}

\paragraph{Use of artificial intelligence.}
The research reported in this article, including computational
exploration, proof strategy, the checking of calculations, and the
preparation and proofreading of the manuscript, was carried out by the
author with the assistance of the AI system ChatGPT (OpenAI).
All mathematical statements, computations, and proofs have been
checked and verified by the author, who takes full responsibility for
the correctness and integrity of the article.

\nocite{*}

\bibliographystyle{amsplain}
\bibliography{references}

\end{document}